\documentclass[12pt]{amsart}
\usepackage{amssymb,amsmath}
\usepackage{xcolor}
\usepackage{enumerate}
\usepackage{url}
\usepackage{stmaryrd}
\usepackage{color}
\title{On curves with one place at infinity
}
\author{Abdallah Assi}
\address{LAREMA, Universit\'e d'Angers,
49045 Angers cedex 01, France}
\email{assi@univ-angers.fr}

\author{Wael Mahboub}
\address{Lebanese university, Faculty of sciences 2, Campus Pierre Gemayel, Fanar, P.O. Box 90656 Jdeidet, Lebanon}
\email{wmahboub@gmail.com}

\newtheorem{teorema}{Theorem}[section]

\newtheorem{proposicion}[teorema]{Proposition}
\newtheorem{lema}[teorema]{Lemma}
\newtheorem{definicion}[teorema]{Definition}

\newtheorem{corolario}[teorema]{Corollary}

\newtheorem{nota}[teorema]{Remark}
\newtheorem{exemple}[teorema]{Example}
\newtheorem{exemples}[teorema]{Examples}

\newenvironment{demostracion}[1]{\paragraph{\sl Proof#1}}{}
\newenvironment{demostracione}[1]{\paragraph{\sl Proof of Theorem 3.1.#1}}{}

\begin{document}
\maketitle

\noindent{\bf Abstract:}  Let ${\mathbb K}$ be an algebraically closed field of characteristic zero. Many of the geometric properties of the plane curves in ${\mathbb K}^2$ with one place at infinity can be obtained from the arithmetic of numerical semigroups associated with these curves. In this survey, we recall the main properties of these semigroups as well as the geometric results obtained from these properties. We also address some problems. 
\medskip

\section*{Introduction and notations}
 
\medskip

\noindent Let ${\mathbb K}$ be an algebraically closed field of characteristic zero, and let $f$ be a polynomial of ${\mathbb
  K}[x,y]$. For all $\lambda\in {\mathbb K}$, we set $f_{\lambda}=f-\lambda$. Hence, we get a family of polynomials $(f_{\lambda})_{\lambda\in {\mathbb K}}$. We suppose that $f_{\lambda}$ is a reduced polynomial for all $\lambda\in {\mathbb K}$. Given a polynomial $P\in {\mathbb K}[x,y]$, we set $C_{P}=V(P)=\lbrace (a,b)\in{\mathbb K}^2|P(a,b)=0\rbrace$ for the plane curve associated with $P$.  Let $h_f(x,y,u)=u^nf(\dfrac{x}{u},\dfrac{y}{u})$, and let $\tilde{C}_f=V(h_f)$ be the projective curve in ${\mathbb P}_{\mathbb K}^2$ obtained by adding the line at infinity $u=0$ (whence, if $P=(a,b)$ is a point in the affine plane, then its projective coordinates are $(a:b:1)$. In particular, the affine plane is identified to ${\mathbb P}_{\mathbb K}^2$ minus the line at infinity $u=0$).  Write $f=\sum_{i=0}^n f_i$ where $f_i$ is homogeneous for all $i\in\lbrace 1,\cdots,n\rbrace$. We have $h_f=\sum_{i=0}^n u^{n-i}f_i$. The points of $f$ at infinity are defined by $u=0,f_n=0$. Given a point $p$ at infinity, we define the number of places of $f$ at $p$ to be the number of irreducible components of the local equation of $h_f$ at $p$. As the number of points at infinity is finite, the number of places at infinity, which is the sum of the number of places at these points, is finite. We denote this number by $\xi_{\infty}(f)$, and we call it the number of places of $f$ at infinity. We define in a similar way the number of places of $f$ at a point $p\in C_f$, and we denote it by $\xi_p(f)$.
  
  
  \noindent Let $g$ be a non-zero polynomial of ${\mathbb K}[x][y]$. We define the intersection
multiplicity of $f$ with $g$, denoted int$(f,g)$, to be the dimension of the ${\mathbb K}$-vector space $\displaystyle{{{{{\mathbb
    K}[x][y]}}\over {(f,g)}}}$. Let $p=(a,b)\in
C_f\cap C_g$. Modulo the translation $X=x-a,Y=y-b$, we may assume that $p=(0,0)$. We define the intersection multiplicity of $f$ with $g$ at $p$, denoted int$_p(f,g)$, to be the dimension of the ${\mathbb K}$-vector space $\displaystyle{{{{{\mathbb
    K}[[x,y]]}}\over {(f,g)}}}$. Note that int$(f,g)=\sum_{p\in  C_f\cap C_g}{\rm int}_p(f,g)$. 

  \noindent Suppose that $f$ has one place at infinity, or equivalently, that $\xi_{\infty}(f)=1$. It follows that $f$ has one point at infinity, whence $f_n=(ax+by)^n$ for some $(a,b)\in {\mathbb K}^2\setminus\lbrace (0,0)\rbrace$. In particular, modulo the change of coordinates $Y=ax+by, X=x$, we may assume that $f_n=y^n$, whence
  
  $$
  f=y^n+a_1(x)y^{n-1}+\cdots+a_n(x)
  $$
  
  \noindent with ${\mathrm deg}_x(a_i(x)<i$ for all $i$ such that $a_i(x)\not=0$ (here ${\mathrm deg}$ denotes the degree). 
  
  \medskip 
  
  \noindent Curves with one place at infinity cover a large class of algebraic plane curves. For instance, let $x(t),y(t)\in{\mathbb K}[t]$, and let $f(x,y)$ be a generator of the kernel of the map $\phi:{\mathbb K}[x,y]\longmapsto {\mathbb K}[t], \phi(x)=x(t), \phi(y)=y(t)$. Then $f$ has one place at infinity. Such a polynomial is also called a polynomial curve. In fact, the development of the theory of curves with one place at infinity was motivated by the following problem, related to polynomial curves: let $n={\rm deg}_tx(t)$ and $m={\rm deg}_ty(t)$. Suppose that ${\mathbb K}[x(t),y(t)]={\mathbb K}[t]$, then, either $n$ divides $m$, or $m$ divides $n$. Segre (\cite{segre}) and Canals and Luis (\cite{canals-luis}) proposed two proofs of this result, but their proofs were wrong. In \cite{a-m}, S.S. Abhyankar and T.T. Moh proposed a proof based on the results of \cite{a-m1} and \cite{a-m2} where they developed the theory of semigroups associated with curves with one place at infinity. Their proof is algebraic and uses the arithmetic of these semigroups.  

  \noindent The aim of this survey is to recall the main properties of curves with one place at infinity, with a focus on the arithmetic properties of semigroups associated with these curves, and their algebraic and geometric applications. The paper is organized as follows: in Section 1 we define  the semigroup associated with a curve with one place at infinity as well as the set of its numerical invariants. Then we give the arithmetic properties of this semigroup. In Section 2 we recall the criterion of Abhyankar for deciding if a polynomial has one place at infinity. This criterion, based on the notion of approximate roots, gives a procedure for an algebraic classification of curves with one place at infinity. This is what we describe in Section 3. Section 4 is devoted to some algebraic and geometric applications. We give first a proof of Abhankar-Moh result. Then we focus on polynomial curves. We prove Then we prove that a family $(f_{\lambda})_{\lambda}$ of curves with one place at infinity contains at most two polynomial curves, and we give an example of a polynomial curve with two non equivalent embedding in the plane. This gives a counter example to a conjecture of V. Shpilrain  and J.-T. Yu (see \cite{V-Y}).
  
  \noindent We stress on the algebraic character of this survey, precisely the semigroup point of view. A geometric proof of Abhyankar-Moh Theorem can be found, among others, in \cite{artal} and \cite{evelia}.  
  \section{Numerical semigroups associated with one place curves}

\medskip

\noindent Let the notations be as in the introduction, in particular $f=y^n+a_1(x)y^{n-1}+\ldots+a_n(x)$ with ${\mathrm deg}_xa_i(x)<i$ for all $i$ such that $a_i(x)\not=0$. Let $U={\mathbb K}[x,y]\setminus (f)$. We have the following.

\begin{itemize}
    \item For all $g_1,g_2\in U, {\rm int}(f,g_1g_2)={\rm int}(f,g_1)+{\rm int}(f,g_2)$.
    \item int$(f,1)=0$. 
\end{itemize}

\noindent In particular, the set 

$$
S=\Gamma(f)=\lbrace {\rm int}(f,g)|g\in U\rbrace
$$

\noindent is a submonoid of ${\mathbb N}$. We will see later that if $\xi_{\infty}(f)=1$, then $S$ is a numerical semigroup, i.e. ${\mathbb N}\setminus S$ is a finite set. 

\subsection{Meromorphic curves}

\medskip

\noindent Let $f$ be as above. We set

$$
F(X,y)=f(X^{-1},y)\in {\mathbb K}[X^{-1}][y]\subset {\mathbb K}((X))[y].
$$

\noindent We get this way a polynomial in one variable over the field of meromorphic series in $X$. Write

$$
F(X,y)=\prod_{i=1}^{\xi(F)}F_i
$$

\noindent where $F_i$ is irreducible in ${\mathbb K}((X))[y]$ for every $i\in\lbrace 1,\cdots,\xi(F)\rbrace$. We have the following.

\medskip

\begin{proposicion}\label{finite-infinite} $F(X,X^{-1}y)=X^{-n}f_{\infty}(X,y)$.  In particular, $\xi(F)=\xi_{\infty}(f)$.
    
\end{proposicion} 
\begin{proof}Easy exercise.
    \end{proof}

\noindent Suppose that $\xi_{\infty}(f)=1$. From Proposition \ref{finite-infinite}, it follows that $F$ is irreducible in ${\mathbb K}((X))[y]$. Let 

$$
y(x)=\sum_{p\geq p_0}c_px^{\dfrac{p}{n}}
$$

\noindent be a root of $F(X,y)=0$ in ${\mathbb K}((X^{1/n}))$. If $w$ is a primitive root of unity in ${\mathbb K}$, then we have

$$
F(t^n,y)=\prod_{k=1}^n (y-y(w^kx)).
$$

\noindent We set Root$(F)=\lbrace y(w^kx)|k=1,\cdots,n\rbrace$.

\medskip

\noindent {\bf Newton-Puiseux exponents.} Set Supp$(y(x))=\lbrace p|c_p\not=0\rbrace$. Clearly, Supp$(y(w^kx))={\rm Supp}(y(x))$ for all $k\in\lbrace 1,\cdots,n\rbrace$. We denote this set by Supp$(F)$. Let GCD stands for the greatest common divisor. We define the set of Newton-Puiseux exponents of $F$ as follows:

\medskip

\noindent $d_1=n$, $m_1={\rm inf}\{p\in{\rm Supp}(F), p\nmid n\}$,  $d_2={\rm GCD}(n,m_1)$, and for all $k\geq 2$

$$
m_k={\rm inf}\{p\in{\rm Supp}(F)| p>m_{k-1}, p\nmid d_k\}, d_{k+1}={\rm GCD}(m_k,d_k).
$$

 \noindent As $F$ is irreducible, it follows that it is the minimal polynomial of $y(X)$ over ${\mathbb K}((X))$, whence GCD$(n,{\rm Supp}(F))=1$ (otherwise, if GCD$(n,{\rm Supp}(F))=d>1$, then the minimal polynomial of $y(X)$ would be of degree $\dfrac{n}{d}$, which is a contradiction). It follows that there exists $h\geq 1$ such that $d_{h+1}=1$. We define the set of Newton-Puiseux exponents of $F$ to be the set NP$(F)=\lbrace m_1,\cdots,m_h\rbrace$. We also set $e_k=\dfrac{d_k}{d_{k+1}}$ for all $k\in\lbrace 1,\cdots,h\rbrace$. 

 \medskip

 \noindent Let $a\in{\mathbb N}$. Given $H(X)\in{\mathbb K}((X^{1/a}))\setminus\lbrace 0\rbrace$, we denote by $O_X(H)$ the order of $H$ with respect to $X$. Given two elements $H_1\not=H_2\in{\mathbb K}((X^{1/a}))\setminus\lbrace 0\rbrace$, we set ${\rm c}(H_1,H_2)=O_X(H_1-H_2)$, and we call it the contact of $H_1$ with $H_2$. The next Lemma characterizes the set of Newton-Puiseux exponents in terms of contacts between the Roots of $f$.

 \begin{lema} \label{contact} (See \cite{a1}) Set Root$(F)=\{y_1,\cdots,y_n\}$. With the notations above we have the following.

 \begin{itemize}
     \item For all $i\geq 1$, there exists $k\in\lbrace 2,\cdots,h\rbrace$ such that ${\rm c}(y_1,y_k)=m_i/n$.

     \item The cardinality of $\lbrace y_k|{\rm c}(y_1,y_k)=m_i/n\rbrace$ is $d_i-d_{i+1}$.
     
 \end{itemize}
     
 \end{lema}
\begin{proof} Let $X=t^n$, and for all $k\in\lbrace 1,\cdots,n\rbrace$, let $Y_k=y(t^n)$. Write $Y_1(t)=\sum_{p\geq p_0}c_pt^p$. For all $k\geq 2$, we have $Y_k(t)=Y_1(w^kt)$, where $w$ is a primitive $n$-th root of unity in ${\mathbb K}$. We have $Y_1-Y_k=\sum_{p\geq p_0}(1-w^p)t^p$. It follows that $O_t(Y_1-Y_k)=s$ if and only if $w^p=1$ for all $p<s$, whence $s\in \lbrace m_1,\cdots,m_h\rbrace$. Let $s=m_k$. If $w^p=1$ for all $p<s$, and $w^s\not=1$, then $w^{d_k}=1$ and $w^{d_{k+1}}\not=1$. This finishes the proof.
\end{proof}

\medskip 

\noindent {\bf Generators of $S=\Gamma(f)$}. Next we shall see how to calculate a set of generators of $S=\Gamma(f)$. We first introduce the sequence $(R_k)_{0\leq k\leq h}$ as follows: $R_0=-n, R_1=m_1$, and for all $k\geq 1$

$$
 R_{k+1}=R_ke_k+m_k-m_{k-1}.
$$
 
\noindent We easily verify that $d_k={\rm GCD}(R_{0},\cdots,R_{k-1})$ for all $k\geq 1$. 

\medskip 

\noindent Given $G\in {\mathbb K}((X))[y]$, we define the intersection of $F$ with $G$ by 

$$
{\rm Int}(F,G)=O_X(\prod_{y(X)\in {\rm Root}(F)}G(X,y(X))).
$$

\noindent Note that $O_X(G(X,y(X))$ does not depend on the choice of the root $y(X)$ of $F$. In particular, Int$(F,G)=nO_X(G(X,y_1(X))$. It also follows from the definition that ${\rm Int}(F,G)=O_X({\rm Res}_y(F,G))$, where Res stands for resultant. 

\noindent Given $h\in{\mathbb K}[z]\setminus\lbrace 0\rbrace$, we denote by ${\rm deg}_zh$ the degree in $z$ of $h$. The next Lemma relates the intersection of polynomials of ${\mathbb K}[x][y]$ to the intersection of their associated meromorphic polynomials.

\begin{lema} \label{int-Int} Let $g\in {\mathbb K}[x,y]$, and let $G(X,y)=g(X^{-1},y)\in {\mathbb K}[X^{-1},y]\subseteq {\mathbb K}((X))[y]$. We have int$(f,g)=-{\rm Int}(F,G)$.
    
\end{lema}
\begin{proof} We have ${\rm Int}(F,G)=O_X({\rm Res}_y(F,G))=-{\rm deg}_x({\rm Res}_y(f,g))=-{\rm int}(f,g)$.
\end{proof}

\noindent It follows from Lemma \ref{int-Int} that, if we set 

$$
U_1=\lbrace G(X,y)=g(X^{-1},y)|g\in {\mathbb K}[x,y]\setminus(f)\rbrace,
$$

\noindent then the set $\lbrace {\rm Int}(F,G)| G\in U_1\rbrace$ is a submonoid of $-{\mathbb N}$, which is nothing but $-S=\lbrace -s|s\in S\rbrace$. Let $r_k=-R_k$ for all $k\in\lbrace 0,\cdots,h\rbrace$. Next, we shall prove that $-S=\langle R_0,\cdots,R_h\rangle$ as a submonoid in $-{\mathbb N}$, which would imply that $S=\langle r_0,\cdots,r_h\rangle$ as a submonoid in ${\mathbb N}$.

\medskip 

\noindent {\bf Pseudo-roots.} Let $k\in\lbrace 1,\cdots,h\rbrace$, and let $\bar{y}^k(X)=\sum_{p<m_k}c_pX^{p/n}$. Let $G_k(X,y)$ be the minimal polynomial of $\bar{y}^k$ over ${\mathbb K}((X))$. As GCD$(n,{\rm Supp}(\bar{y}^k))=d_k$, it follows that deg$_yG_k=\dfrac{n}{d_k}$. We say that $G_k$ is the $k$-th pseudo-root of $F$. With these notations, we have the following.

\begin{proposicion} (see \cite{a1}) For all $k\in\lbrace 1,\cdots,h\rbrace$, the polynomial $G_k$ is irreducible. Moreover, for all $k\geq 2$, the following conditions hold.
\begin{enumerate}
    \item ${\rm Int}(F,G_k)=nO_X(G_k(X,y(X)))=R_k$.
    \item The set of Newton-Puiseux exponents of $G_k$ is given by ${\rm NP}(G_k)=\lbrace m_1/d_{k},\cdots,m_{k-1}/d_k\rbrace$.
    \item For all $k\geq 2$, $\lbrace G_1,\cdots,G_{k-1}\rbrace$ is the set of pseudo-roots of $G_k$, and for all $i\in\lbrace 1,\cdots,k-1\rbrace, {\rm Int}(G_k,G_i)=R_i/d_k$.
\end{enumerate}
\end{proposicion}

\noindent Let $\underline{G}=(G_1,\cdots,G_h)$ be the sequence of pseudo-roots of $F$. Given $G\in U_1$, we can write 

$$
G=\sum_{\underline{\theta}}c_{\underline{\theta}}(X)G_1^{\theta_1}\cdots G_h^{\theta_h}F^{\theta_{h+1}}
$$

\noindent with $0\leq \theta_i<e_i$ for all $i\in \lbrace 1,\cdots,h\rbrace$. This expression is unique (see \cite{a1}). We call it the $\underline{G}$-adic expansion of $G$. As $g\in {\mathbb K}[x,y]\setminus (f)$, there is at least one $\underline{\theta}$ such that $\theta_{h+1}=0$. For such a $\underline{\theta}$, we have 

$$
{\rm Int}(F,c_{\underline{\theta}}(X)G_1^{\theta_1}\cdots G_h^{\theta_h})=O_Xc_{\theta}(X)n+\sum_{i=1}^h\theta_iR_i
$$

\noindent The following technical lemma shows that two distinct monomials in the $\underline{G}$-adic expansion of $G$ intersect $F$ differently.

\begin{lema} \label{u} Let $N=\lbrace \underline{\theta}=(\theta_0,\theta_1,\cdots,\theta_h)\in{\mathbb N}^{h+1}|0\leq \theta_i<e_i$ for all $i\geq 1\rbrace$. For all $\underline{\alpha},\underline{\beta}\in N$, if $\underline{\alpha}\not=\underline{\beta}$, then $\alpha_0 n+\sum_{i=1}^h\alpha_iR_i\not=\beta_0 n+\sum_{i=1}^h\beta_iR_i$.
    
\end{lema}
\begin{proof} Suppose that $\alpha_0 n+\sum_{i=1}^h\alpha_iR_i=\beta_0 n+\sum_{i=1}^h\beta_iR_i$ for some $\underline{\alpha}\not=\underline{\beta}$ in $N$. Let $j$ be the greatest integer such that $\alpha_j\not=\beta_j$, and suppose, without loss of generality, that $\beta_j>\alpha_j$. We have $(\alpha_0-\beta_0) n+\sum_{i=1}^{j-1}(\alpha_i-\beta_i)R_i=(\beta_j-\alpha_j)R_j$. Dividing this equality by $d_{j+1}$ we get

$$
(\alpha_0-\beta_0){\dfrac{n}{d_{j+1}}}+\sum_{i=1}^{j-1}(\alpha_i-\beta_i){\dfrac{R_i}{d_{j+1}}}=(\beta_j-\alpha_j)\dfrac{R_j}{d_{j+1}}. 
$$

\noindent As GCD$(n,d_1,\cdots,d_{j-1})=d_j$, it follows that $(\beta_j-\alpha_j)\dfrac{R_j}{d_{j+1}}$ divides $e_j=\dfrac{d_j}{d_{j+1}}$, but GCD$(e_j,\dfrac{R_j}{d_{j+1}})=1$, and $0<\beta_j-\alpha_j<e_j$. This is a contradiction.\end{proof}

\noindent It follows from Lemma \ref{u} that there exists a unique $\underline{\theta}$ such that if $\theta_0=O_X(c_{\theta}(X))$, then Int$(F,G)=\theta_0n+\sum_{i=1}^h\theta_iR_i$, whence $-S=\langle R_0=-n,R_1,\cdots,R_h\rangle$. This implies that

$$
S=\langle r_0=n,r_1,\cdots,r_h\rangle.
$$

\noindent As GCD$(r_0,\cdots,r_h)=1$, it follows that $S$ is a numerical semigroup. We set $\underline{r}=(r_0,r_1,\cdots,r_h)$, and we call it the $\underline{r}$-sequence associated with $f$. We set $\underline{d}=(d_1,\cdots,d_h,d_{h+1}=1)$, and we call it the $\underline{d}$-sequence associated with $f$. We set $\underline{e}=(e_1,\cdots,e_h)$, and we call it the $\underline{e}$-sequence associated with $f$. We set $\underline{m}=(m_1,\cdots,m_h)$, and we call it the $\underline{m}$-sequence associated with $f$. We will sometimes refer to $h$ as $h(f)$.

\medskip

\noindent The above results show that the calculation of the set of generators of $S$ requires the calculation of the set of Newton-Puiseux exponents of $F$, hence the calculation of a root $y(X)$ of $F(X,y)=0$. Next, we shall see later that this calculation is possible directly from the equation of $f$, using the notion of approximate roots of $f$.

\medskip 

\noindent Next, we shall give some arithmetic properties of $S=\Gamma(f)$. We shall need the following Lemma.

\begin{lema}\label{in-S} With the notations above, for all $k\in\lbrace 1,\cdots,h\rbrace, e_kr_k\in\langle r_0,\cdots,r_{k-1}\rangle$, and $e_k$ is the smallest integer with this property.
\end{lema}

\begin{proof} We shall prove the first assertion by induction on $k$. If $k=1$, then $e_1R_1=(d_1/d_2)R_1=(n/d_2)R_1=(R_1/d_2)n=(R_1/d_2)(-R_0)$, whence $e_1r_1\in\langle r_0\rangle$. Suppose that the assertion is true for $j\in\lbrace 1,\cdots,k-1\rbrace$, and Write 

$$
G_{k+1}=G_{k}^{e_{k}}+\alpha_1 G_{k}^{e_k-1}+\cdots+\alpha_{e_k}
$$ 

\noindent with ${\rm deg}_y\alpha_i<n/d_k$ for all $i\in\lbrace 1,\cdots,e_k\rbrace$ such that $\alpha_i\not=0$. In particular, 

$$
{\rm Int}(G_{k+1},\alpha_i)=\sum_{j=0}^{k-1}\theta^i_j\dfrac{R_j}{d_{k+1}}.
$$

\noindent A similar argument as in Lemma \ref{u} shows that for all $i\not=j$, if $\alpha_i\not=0$ and $\alpha_j\not=0$, then ${\rm Int}(G_{k+1},\alpha_iG_k^{e_k-i})\not={\rm Int}(G_{k+1},\alpha_jG_k^{e_k-j})$. Moreover, for all $i\in\lbrace 1,\cdots,e_k-1\rbrace$, if $\alpha_i\not=0$, then ${\rm Int}(G_{k+1},G_k^{e_k})\not={\rm Int}(G_{k+1},\alpha_iG_k^{e_k-i})$. As ${\rm Int}(G_{k+1},G_{k+1})$ is not an integer, this forces $e_kr_k$ to be equal to ${\rm Int}(G_{k+1},\alpha_{e_k})$, whence $e_k(R_k/d_{k+1})=\sum_{j=1}^{k-1}\theta_i(R_i/d_{k+1})$ with $\theta_0,\cdots,\theta_k\in{\mathbb N}$, and consequently $e_kr_k=\sum_{j=1}^{k-1}\theta_ir_i$. Now the same argument as in Lemma \ref{u} shows that for all $i\in\lbrace 1,\cdots,e_k-1\rbrace, ir_k\notin \langle r_0,\cdots,r_k\rangle$. This finishes the proof. 
\end{proof}

\noindent Let the notations be as above and let $s\in{\mathbb Z}$. As $d_{h+1}=1$, there exists $(\theta_0,\cdots,\theta_h)\in{\mathbb Z}^{h+1}$ such that

$$
s=\sum_{i=0}^h\theta_ir_i.
$$

\noindent If $\theta_h\notin [0,e_h[$, then we write $\theta_h=a_he_h+b_h$ with $0\leq b_h< e_h$. As $e_hr_h\in\langle r_0,\cdots,r_{h-1}\rangle$, we have $s=\sum_{i=0}^{h-1}\theta^1_ir_i+b_hr_h$ with $\theta_0^1,\cdots,\theta_{h-1}^1\in{\mathbb Z}$. Then an easy induction shows that 

$$
s=\alpha_0r_0+\sum_{i=1}^h\alpha_ir_i
$$

\noindent with $\alpha_0\in{\mathbb Z}$ and $0\leq \alpha_i<e_i$ for all $i\in\lbrace 1,\cdots,h\rbrace$. Moreover, a similar argument as in Lemma \ref{u} shows that this expression is unique. We call it the standard representation of $s$.

 \begin{proposicion} \label{arithmetic} (See \cite{Sat}) Let $s\in {\mathbb Z}$ and let $s=\alpha_0r_0+\sum_{i=1}^h\alpha_ir_i$ be the standard representation of $s$. We have the following.
     
     \begin{enumerate}
     \item $s\in S=\Gamma(f)$ if and only if $\alpha_0\geq 0$.

     \item Let $a,b\in{\mathbb Z}$, and suppose that $a+b=\sum_{i=1}^h(e_i-1)r_i-n$. We have $a\in S$ if and only if $b\notin S$. In particular $\sum_{i=1}^h(e_i-1)r_i-n$ is an odd integer.
      
      \item Let $F(S)=(\sum_{i=1}^h(e_i-1)r_i)-n$. For all $s>F(S), s\in S$, i.e. $F(S)$ is the Frobenius number of $S$ (equivalently, $F(S)$ is the largest integer not in $S$). We set $C(S)=F(S)+1$, and we call it the conductor of $S$.
 \end{enumerate}
    
\end{proposicion}

\begin{proof} Clearly, if $\alpha_0\geq 0$, then $s\in S$. Conversely, if $s\in S$, then $s=\sum_{i=0}^h\theta_ir_i$ with $\theta_0,\cdots,\theta_h\in{\mathbb N}$. If $\theta_h\geq e_h$, then we write $\theta_h=a_he_h+b_h$ with $a_h\in{\mathbb N}$ and $0\leq b<e_h$. By Lemma \ref{in-S}, $e_hr_h\in\langle r_0,\cdots,r_{h-1}\rangle$, whence $s=\sum_{i=0}^{h-1}\theta^1_ir_i+b_hr_h$ with $\theta^1_i\geq 0$ for all $i\in\lbrace 0,\cdots,h-1\rbrace$. Then 1)  follows by an easy induction.

\medskip
\noindent To prove 2), let $a=\sum_{i=0}^h\alpha_ir_i$ (resp. $b=\sum_{i=0}^h\beta_ir_i)$ be the standard representation of $a$ (resp. $b$). We have

$$
\sum_{i=0}^h(\alpha_i+\beta_i)r_i=-n+\sum_{i=}^h(e_i-1)r_i.
$$

\noindent If $\alpha_h+\beta_h\geq e_h$, then, as $0\leq \alpha_h+\beta_h\leq 2e_h-2$, the Euclidean division of $\alpha_h+\beta_h$ by $e_h$ would not give $e_h-1$ as a remainder, which is a contradiction. Whence $\alpha_h+\beta_h=e_h-1$, and an easy induction shows that $\alpha_i+\beta_i=(e_i-1)$ for all $1\leq i\leq h-1$. We finally get $\alpha_0+\beta_0=-n$, whence

$$
a\in S\Longleftrightarrow \alpha_0\geq 0\Longleftrightarrow \beta_0<0 \Longleftrightarrow b\notin S.
$$

\noindent This finishes the proof.

\medskip 

\noindent 3) As $\sum_{i=1}^h(e_i-1)r_i-n+1$ is the standard representation of $F(S)$, it follows from 1) that $F(S)\notin S$. On the other hand, for all $s>F(S)$, $F(S)=s+F(S)-s$. But $F(S)-s<0$, hence $F(S)-s\notin S$, and by 2), $s\in S$. Finally suppose that $F(S)$ is an even number. We have $F(S)=(F(S)/2)+(F(S)/2)$ and $F(S)/2\notin S$. This contradicts 2).\end{proof}

\medskip

\begin{nota} Part 2) of Proposition \ref{arithmetic} implies that $S$ is a symmetric numerical semigroup (i.e. the cardinality of $g(S)$ of ${\mathbb N}\setminus S$, called also the genus of $S$, equals $C(S)/2$).  It follows that $S$ is irreducible. For the definition and properties of symmetric and irreducible numerical semigroups, we refer to \cite{assi3}.
\end{nota}

\section{Approximate roots and Abhyankar criterion} 
\noindent Let the notations be as in Section 2. In particular $f=y^n+a_1(x)y^{n-1}+\cdots+a_n(x)$ is a non zero polynomial with one place at infinity at the point $y=0$. Let $d$ be a divisor of $n$ and let $g$ be a monic polynomial of ${\mathbb K}[x][y]$ of degree $\dfrac{n}{d}$ in $y$. Write

$$
f=g^d+\alpha_1(x,y)g^{d-1}+\cdots+\alpha_d(x,y)
$$

\noindent with deg$_y\alpha_i<i$ for all $i$ such that $\alpha_i\not=0$. We call this expression the expansion of $f$ with respect to $g$. We say that $g$ is a $d$-th approximate root of $f$ if $\alpha_1=0$. Such a polynomial alawys exists, and it is unique (see \cite{a1}, for example). We denote it by App$(f,d)$. We also have the following.

\begin{nota} Let  $i\in\lbrace 1,\cdots,h-1\rbrace$, and let $d_i>d_{i+1}$ be two elements of the $\underline{d}$-sequence associated with $f$. Let $g_k={\rm App}(f,d_k), k=i,i+1$. We have 

$$
g_{i+1}={\rm App}(g_i,e_i=\dfrac{d_i}{d_{i+1}}).
$$
\noindent More precisely, if $g_{i+1}=g_i^{e_i}+\gamma_1g_i^{e_i-1}+\cdots+\gamma_{e_i}$ is the expansion of $g_{i+1}$ with respect to $g_i$, then $\gamma_1=0$.
\end{nota}

\noindent With these notations, we have the following.

\begin{proposicion}\label{r-sequence} (see \cite{a-m1}) Let $\underline{d}= (d_1,\cdots,d_h)$ be the $\underline{d}$-sequence associated with $f$. For all $k\in\lbrace 1,\cdots,h\rbrace$, if $g_k={\rm App}(f,d_k)$, then int$(f,g_k)=r_k$.
\end{proposicion}

\noindent The approximate roots of $f$ can be calculated from the equation of $f$. In particular the sequences associated with $f$ can be calculated in an algorithmic way: suppose, after possibly a change of coordinates, that App$(f,n)=y$. We set  $r_0=d_1=n$, $r_1={\rm int}(f,y)={\rm deg}_xa_n(x)$, and $d_2={\rm GCD}(r_0,r_1)$. Suppose that we already have $(r_0,\cdots,r_k)$ and $(d_1,\cdots,d_{k+1})$. We set $g_{k+1}={\rm App}(f,d_{k+1})$, $r_{k+1}={\rm int}(f,g_{k+1})$, and $d_{k+2}={\rm GCD}(r_{k+1},d_{k+1})$. This process will stop giving the $\underline{r}$-sequence as well as the set of other sequences associated with $f$.

\begin{exemple} \label{ex1} Let $f=y^6-2x^2y^3-xy+x^4$. We have

 \begin{itemize}
  \item $r_0=6=d_1$.

  \item ${\rm App}(f,6)=y, r_1={\rm int}(f,y)=4, d_2=2$.

  \item ${\rm App}(f,2)$: $f=(y^3)^2+(-2x^2)(y^3)+(-xy+x^4)$. We set $G=y^3-x^2$, then we get $f=G^2-xy$, hence
  ${\rm App}(f,2)=G$.

  \item $r_2={\mathrm int}(f,G)={\mathrm int}(G,xy)=5$, and $d_3=1$. Consequently, $(6,4,5)$ is the the $\underline{r}$-sequence of $f$.
  \end{itemize}
  \end{exemple}  
    \begin{nota} The $\underline{r}$-sequence associated with $f$ is not a minimal set of generators of $f$. In the example above, $\Gamma(f)=\langle 6,4,3\rangle=\langle 4,3\rangle$. We will see in Section 4. that the geometry of $C_f$ depends only on the $\underline{r}$-sequence.
        \end{nota}

\begin{nota} In the algorithm above, the intersection int can be calculated using only the expansion of $f$ with respect to its approximate roots. Precisely, let $k\geq 1$ and let 

$$
f=g_k^{d_k}+\alpha^k_2g_k^{d_k-2}+\cdots+\alpha^k_{d_k}
$$

\noindent be the expansion of $f$ with respect to $g_k$. We have ${\rm int}(f,g_k)={\rm int}(g_k,\alpha^k_{d_k})$, whence int$(f,g_k)$ is obtained from the $\underline{g}_{k-1}$-adic expansion of $\alpha^k_{d_k}$, where $\underline{g}_{k-1}=(g_1=y,g_2,\cdots,g_{k-1})$. Let $f=y^6-2x^2y^3-xy+x^4$ be the polynomial of example \ref{ex1}. We have $r_0=6=d_1, r_1=4$, and $d_2=2$. We have ${\rm App}(f,d_2)=y^3-x^2$, and $f=(y^3-x^2)^2-xy$, whence $r_2={\rm int}(f,y^3-x^2)={\rm int}(y^3-x^2,xy)$. But the $\underline{r}$-sequence associated with $y^3-x^2$ is $(r_0/d_2,r_1/d_2)=(3,2)$, hence ${\rm int}(y^3-x^2,xy)=3+2=5$.
\end{nota}

\noindent It follows from the algorithm above that if $f$ has one place at infinity, then we can calculate its $\underline{r}$-sequence (hence the other sequences) from the equation. Now the following question arises: given a polynomial in ${\mathbb K}[x,y]$, can we decide if this polynomial has one place at infinity? The answer is yes, and an algorithm has been given by Abhyankar in \cite{a3}. 

\medskip

\noindent {\bf Deciding if a polynomial has one place at infinity.} Let $P$ be a polynomial of ${\mathbb K}[x,y]$, and suppose, after possibly a change of coordinates, that 

$$
P(x,y)=y^p+b_2(x)y^{p-2}+\cdots +b_p(x).
$$ 

\noindent Let $d$ be a divisor of $p$, and let $g$ be a monic polynomial of degree $\dfrac{n}{d}$ in $y$. Let 

$$
P=g^d+\alpha_1g^{d-1}+\cdots+\alpha_d
$$

\noindent be the expansion of $P$ with respect to $g$. Set $\alpha_0=1$, and let $s={\rm int}(P,g)$, and for all $i$ such that $\alpha_i\not=0$, let $s_i={\rm int}(P,\alpha_i)$ (in particular $s_0=0)$). We define the Newton polygon of $P$ with respect to $g$, denoted ${\rm NP}(P,g)$, to be the set of compact faces of the convex hall of 

$$
\cup_{\alpha_i\not=0} (s_i,(d-i)s)+({\mathbb R}^-\times \lbrace 0\rbrace)
$$

\noindent {\bf The criterion.} Let $P$ be as above and assume that $P(0,0)=0$.  Let $r_0=d_1=p$. If $p=1$ then $P$ has one place at infinity. Suppose that $p\geq 2$.  If $b_p(x)=0$, then $P$ is not irreducible, whence it has at least two places at infinity. Suppose that $b_p(x)\not=0$, then clearly App$(P,p)=y$. Let $r_1=m={\rm deg}(b_p(x))={\rm int}(P,y)$. We set 
$$
d_2={\rm GCD}(d_1,r_1), g_2={\rm App}(f,d_2)
$$

\noindent and $r_2={\rm int}(P,g_2)$. Suppose that we have $g_1,\cdots,g_k$, and $r_0,r_1,\cdots,r_k$. We set 

$$
d_{k+1}={\rm GCD}(r_k,d_k), g_{k+1}={\rm App}(P,d_{k+1})
$$

\noindent and $r_{k+1}={\rm int}(P,g_{k+1})$. We have the following.

\begin{teorema} (See \cite{a3}) Let the notations be as above. The polynomial $P$ has one place at infinity if and only if the following conditions hold.

\begin{itemize}
    \item The sequence $(d_k)_k$ is strictly decreasing, and there exists $h$ such that $d_{h+1}=1$.

    \item Set $P=g_{h+1}$. For all $k\in\lbrace 1,\cdots,h\rbrace$, the Newton polygon NP$(g_{k+1},g_k)$ is the segment of line $((e_k\dfrac{r_k}{d_{k+1}},0), (0,e_k\dfrac{r_k}{d_{k+1}}))$.
\end{itemize}
    
\end{teorema}

\begin{exemples} 1. Let $f=y^6-2x^2y^3-xy+x^4$. It follows from Example \ref{ex1} that $r_0=6=d_1$, $g_1={\rm App}(f,6)=y, r_1={\rm int}(f,y)=4, d_2=2$, $g_2={\rm App}(f,2)=y^3-x^2, r_2={\mathrm int}(f,g_2)=5$, and $d_3=1$. Moreover
\begin{itemize}
   \item $g_2=g_1^3-x^2$, int$(g_2,g_1)=6$, int$(g_2,x^2)=6$, whence ${\rm NP}(g_2,g_1)$ is the segment of line $((0,6), (6,0))$.
   \item $f=g_3=g_2^2-xg_1$, int$(f,g_2^2)=10$, and int$(f,xg_1)=10$, whence ${\rm NP}(f,g_2)$ is the segment of line $((0,10), (10,0))$. 
\end{itemize}

\noindent In particular $f$ has one place at infinity.

\medskip

\noindent 2. Let $f=(y^3-x^2)^2-y$. Similar calculations as above show that $d_1=r_0=6, {\rm App}(f,6)=y$, $r_1=6$, $d_2=2$, App$(f,2)=y^3-x^2$, and int$(f, y^3-x^2)=2$, whence $d_3=d_2=2$. This proves that $f$ has more than one place at infinity.
\end{exemples}

\begin{nota} In the criterion above, if $P$ has one place at infinity, then for all $k\in\lbrace 1,\cdots,h\rbrace$, if 

$$
g_{k+1}=g_k^{e_k}+\alpha^k_2g_k^{e_k-2}+\cdots+\alpha_{e_k}
$$

\noindent is the expansion of $g_{k+1}$ with respect to $g_k$, then for all $i\in\lbrace 1,\cdots,e_k-1\rbrace$, if  $\alpha^k_i\not=0$, then int$(g_{k+1},\alpha^k_i)<i\dfrac{r_k}{d_{k+1}}$. 
    
\end{nota}

\noindent In connection with the criterion above, we address the following problem.

\medskip

\noindent {\bf Problem.} Given $P\in{\mathbb K}[x,y]$, is it possible to calculate the number of places at infinity only from the equation of $P$?

\medskip

\begin{nota} \label{r-sequence} Let $f=y^n+a_1(x)y^{n-1}+\cdots+a_n(x)$, and suppose that $n>{\rm deg}_xa_n(x)$, and also that $f$ has one place at infinity. 

\begin{itemize}
    \item It follows from the criterion above that for all $\lambda\in{\mathbb K}$, $f_{\lambda}=f-\lambda$ has one place at infinity. Moreover, $\Gamma(f)=\Gamma(f_{\lambda})$, and $f_{\lambda}$ has the same sequences as $f$. Also the sequence $(g_1=y,g_2,\cdots,g_h)$ of approximate roots of $f$ is also the sequence of approximate roots of $f_{\lambda}$. This nice property is proper to curves with one place at infinity.

\item  The same criterion shows the following: let $g_k={\rm App}(f,d_k), k\in\lbrace 1,\cdots,h\rbrace$. Then $g_k$ is a monic polynomial of degree $\dfrac{n}{d_k}$ in $y$, with one place at infinity. Moreover, if $S_k=\Gamma(g_k)$ denotes the numerical semigroup of $g_k$, then

$$
S_k=\langle \dfrac{n}{d_k},\cdots,\dfrac{r_{k-1}}{d_k}\rangle
$$
\noindent and $(\dfrac{n}{d_k},\cdots,\dfrac{r_{k-1}}{d_k})$ is the $\underline{r}$-sequence associated with $g_k$.
\end{itemize}
\end{nota}

\noindent The following proposition gives a characterization of $\underline{r}$-sequences associated with  polynomials with one place at infinity.

\begin{proposicion}\label{conditions} Let $f$ be as in Remark \ref{r-sequence}. It follows from the description of the sequences associated with $f$ that $r_kd_k>r_{k+1}d_{k+1}$ for all $k\in\lbrace 1,\cdots,h-1\rbrace$. Moreover, 

\begin{itemize}
    \item $r_{k}\dfrac{d_k}{d_{k+1}}\in\langle r_0,\cdots,r_{k-1}\rangle$.

    \item For all $1\leq i\leq d_k-1$, $ir_{k}\notin\langle r_0,\cdots,r_{k-1}\rangle$.
\end{itemize}

\medskip 
\noindent Conversely, let $(r_0,r_1,\cdots,r_h)$ be a sequence of non negative integers, and let $d_k={\rm GCD}(r_0,\cdots,$ $r_{k-1})$ for all $k\in \lbrace 1,\cdots,h+1\rbrace$. Suppose that the following conditions hold.

\begin{itemize}
    \item $r_0>r_1$ and $d_{h+1}=1$.

    \item $r_kd_k>r_{k-1}d_{k-1}$ for all $k\in\lbrace 2,\cdots,h\rbrace$.

    \item For all $k\in\lbrace 1,\cdots,h\rbrace$, $r_k\dfrac{d_k}{d_{k+1}}\in\langle r_0,\cdots,r_{k-1}\rangle$
    \item For all $1\leq i<\dfrac{d_k}{d_{k+1}}, ir_k\notin \langle r_0,\cdots,r_{k-1}\rangle$.
\end{itemize}

\noindent Then there exists a polynomial $f$ with one place at infinity such that $\underline{r}=(r_0,\cdots,r_h)$ is the $\underline{r}$-sequence associated with $f$.
\end{proposicion}
\begin{demostracion}{.} The first assertions follow from Lemma \ref{in-S} and the equality $r_kd_k=r_{k-1}d_{k-1}+m_k-m_{k-1}$. Conversely, given $(r_0,\cdots,r_h)$ with the properties cited above, we set $g_1=y$, $g_2=y^{n/d_2}-x^{r_1/2}$, and for all $k\geq 2$, if $r_kd_k=\sum_{i=0}^{k-1}a_ir_i$ with $0\leq a_i<e_i$ for all $i$, then we set

    $$
g_{k+1}=g_k^{d_k/d_{k+1}}-x^{a_0}g_1^{a_1}\ldots g_{k-1}^{a_{k-1}}.
    $$

\noindent By Abhyankar criterion, $f=g_{h+1}$ is a polynomial with one place at infinity, and $(g_1,\cdots,g_h)$ is its set of approximate roots. 

\end{demostracion}
\medskip





 
\section{A constructive classification of polynomials with one place at infinity}

\noindent The constructive criterion of Section 2. gives a constructive way to classify the set of polynomials with one place at infinity (see \cite{assi2}). Given a sequence of integers $\underline{r}=(r_0=n, r_1,\cdots,r_h)$ that satisfies the conditions of Proposition \ref{conditions}, there exists a polynomial $f$ with one place at infinity such that $\underline{r}$ is the $\underline{r}$-sequence of $f$. Conversely the $\underline{r}$-sequence associated with a polynomial with one place at infinity satisfies the conditions of Proposition \ref{conditions}. Next we give, using the properties of $\underline{r}$-sequences of polynomials with one place at infinity, and the criterion of Section 2., a constructive classification of these polynomials.

\medskip

\noindent {\bf Classification.} (see \cite{assi2}) Let $f$ be as in Remark \ref{r-sequence}, and consider the sequences associated with $f$. The first approximate root of $f$ being of degree one in $y$, we shall assume that App$(f,n)=y$. Let $d_2={\rm GCD}(r_0,r_1)$ and let $g_2={\rm App}(f,d_2)$. As $\Gamma(g_2)=\langle n/d_2,r_1/d_2\rangle$, and using the properties of NP$(g_2,y)$, we can write 

$$
g_2={\rm App}(f,d_2)=y^{n/d_2}+a_2x^{r_1/d_2}+\sum c^{2}_{ij}x^iy^j
$$

\noindent  with $a_2\not=0$ and for all $(i,j)$, if $c^2_{ij}\in{\mathbb K}\setminus\lbrace 0\rbrace$, then $i\dfrac{r_0}{d_2}+j\dfrac{r_1}{d_2}<\dfrac{r_0}{d_2}\dfrac{r_1}{d_2}$. 

\noindent Suppose that $r_1$ divides $r_0$. We have $d_1=r_0$ and $d_2=r_1$, whence $\Gamma(g_2)=\langle r_0/d_2,1\rangle$. In particular, it follows from the inequality 

$$
i\dfrac{r_0}{d_2}+j\dfrac{r_1}{d_2}=i\dfrac{r_0}{d_2}+j<\dfrac{r_0}{d_2}\dfrac{r_1}{d_2}=\dfrac{r_0}{d_2}
$$

\noindent that $(i,j)=(0,0)$, and consequently $g_2=y^{n/d_2}+a_2x+b_2$ with $a_2\not=0$. In particular ${\mathbb K}[g_2,y]={\mathbb K}[x,y]$. Let $\sigma:{\mathbb K}[x,y]\longmapsto {\mathbb K}[X,Y]$ be the automorphism defined by $\sigma(g_2)=Y,\sigma(y)=X$ (whence $\sigma(x)=(Y-X^{n/d_2}-b_2)/a_2$ and $\sigma(y)=X$). Let $F(X,Y)=\sigma(f)$. Then $F$ has one place at infinity. Moreover, as the intersection multiplicity is an invariant modulo $\sigma$, the following holds.

\begin{enumerate}
    \item For all $g\in{\mathbb K}[x,y]\setminus\lbrace 0\rbrace$, we have ${\rm int}(f,g)={\rm int}(F,\sigma(g))$.
    \item ${\rm int}(F,\sigma(g_k))=r_k$ for all $k\in\lbrace 2,\cdots,h\rbrace$. 
\end{enumerate}

\medskip

\noindent In particular, the $\underline{r}$-sequence associated with $F$ is given by $(r_1,\cdots,r_h)$. Note that $r_1={\rm deg}_YF$ and $r_2={\rm deg}_XF(X,0)$, if $r_1<r_2$, then we consider $\tilde{F}(X,Y)=F(Y,X)$. From this we may assume that $r_1>r_2$. If $r_2$ divides $r_1$, then we apply to $F$ the same argument as for $f$. This motivates the following definition.

\begin{definicion}\label{r-reduced} With the notations above, we say that $f$ is reduced if the $\underline{r}$-sequence $(r_0,\cdots,r_h)$ associated to $f$ satisfies the following conditions.
\begin{enumerate}
    \item $r_0>r_1$.
    \item $r_1$ does not divide $r_0$.
\end{enumerate}
    
\end{definicion}

\noindent Then we get the following.

\begin{lema} Given a polynomial with one place at infinity, there exists an automorphism $\sigma$ of ${\mathbb K}^2$ such that $\sigma(f)$ is reduced.
    
\end{lema}

\begin{exemple} Let $f=(y^2-x)^2-xy$. Then $\Gamma(f)=\langle 4,2,3\rangle$. Let $\sigma: {\mathbb K}[x,y]\longmapsto {\mathbb K}[X,Y]$ be the automorphism such that $Y=y^2-x, X=y$, then $\sigma(f)=Y^2-X(X^2-Y)=Y^2+XY-X^3$ whose $\underline{r}$-sequence is $(2,3)$. Interchanging $X$ with $Y$, and multiplying by $-1$, $\sigma(f)$ becomes $Y^3-XY+X^2$ whose $\underline{r}$-sequence is $(3,2)$, and $Y^3-XY+X^2$ is reduced.
    \end{exemple}
    
\noindent We shall next see how to construct the set of equations of reduced polynomials with one place at infinity (see \cite{assi3}). Let $(r_0,\cdots,r_h)$ be as in Definition \ref{r-reduced}. We set $r_0=n,r_1=m, g_1=y$, and

$$
g_2=y^{n/d_2}+a_2(x)y^{n/d_2-2}+\cdots+a_{n/d_2}(x)
$$

\noindent with the following conditions.
\begin{enumerate}
    \item For all $i\in\lbrace 2,\cdots,(n/d_2)-1\rbrace$, if $a_i(x)=\sum c^i_{k}x^k$, then $k(n/d_2)<i(m/d_2)$ for all $k$ such that $c^i_k\not=0$.
    \item deg$_x(a_{n/d_2}(x))=m/d_2$.
\end{enumerate}

\noindent Suppose that we already have $(g_1,\cdots,g_k)$, and let $g_{k+1}={\rm App}(f,d_{k+1})$. As $\Gamma(g_{k+1})=\langle r_0/d_{k+1},\cdots,r_k/d_{k+1}\rangle$, and by the properties of NP$(g_{k+1},g_k)$, if $r_ke_k=a_0^kr_0+\cdots+a_{k-1}^k r_{k-1}$ is the standard representation of $r_ke_k$, and if $U_k=\lbrace (\alpha_1,\cdots,\alpha_k)\in {\mathbb N}^k\mid 0\leq \alpha_i<e_i$ for all $1\leq i\leq k\rbrace$, then 

$$
g_{k+1}=g_k^{e_k}+c^{k+1}x^{a_0}g_1^{a_1}\cdots g_{k-1}^{a_{k-1}}+\sum_{\underline{\theta}^{k+1}\in U_k} c^{k+1}_{\underline{\theta}^{k+1}}x^{\theta^{k+1}_0}g_1^{\theta^{k+1}_1}\cdots g_k^{\theta^{k+1}_k}
$$

\noindent with $c^{k+1}\in{\mathbb K}\setminus\lbrace 0\rbrace$, and for all $\underline{\theta}^{k+1}$, if $c^{k+1}_{\underline{\theta}^{k+1}}\in{\mathbb K}\setminus\lbrace 0\rbrace$, then $\sum_{j=1}^k\theta_0^{k+1}r_0+\cdots+\theta_k^{k+1}r_k<r_ke_k$. 


\begin{exemple} Let $\underline{r}=\langle 6,4,5\rangle$. We set $g_1={\rm App}(f,6)=y$. As $d_2=2$, $\Gamma(g_2)={\rm App}(f,2)$ satisfies $\Gamma(g_2)=\langle 3,2\rangle$, whence

$$
g_2=y^2+ax^3+bxy+cy+dx+e
$$
\noindent with $a\not= 0$. Now $e_2r_2=10=4+6$. Finally 

$$
f=g_2^2+a_1xy+\sum c_{ij}x^iy^j
$$

\noindent with $a_1\in{\mathbb K}\setminus\lbrace 0\rbrace$ and for all $(i,j)$, $c_{ij}\in{\mathbb K}$, and $6i+4i<10$ (whence $(i,j)\in\lbrace (1,0),(0,1),(0,0)\rbrace$).
\end{exemple}

\medskip
 \begin{nota} The above classification can be refined in terms of the conductors of the given semigroups. More precisely, let $c$ be an even non-negative integer, then $c$ is the conductor of the semigroup of a polynomial with one place at infinity (for example, if $c=2p$, then $c$ is the conductor of the semigroup generated by $2p+1$ and $2$, and $(2p+1,2)$ is the $\underline{r}$-sequence of the polynomial $f=y^{2p+1}-x^2$, which has one place at infinity. For a fixed even integer $c$, the number of $\underline{r}$-sequences of polynomials with one place at infinity with $c$ as a conductor is finite. Using the above algorithm, the classification can be generalized into a construction of reduced equations of polynomials with one place at infinity with a fixed conductor for the associated semigroups. An algorithm has been given in \cite{assi4} and implemented in GAP (see \cite{del}).
 \end{nota}

 \noindent We end this section with the following result which will be used in Section 4.

\begin{proposicion}\label{aut} Let $f$ be a polynomial with one place at infinity, and let $(r_0=n,r_1,\cdots,r_h)$ be the $\underline{r}$-sequence associated with $f$, and suppose that $r_0>r_1$. For all $k\in\lbrace 1,\cdots,h\rbrace$, let $g_k={\rm App}(f,d_k)$. If the conductor $C(f)=0$, then the following conditions hold.

\begin{enumerate}
\item For all $k\in\lbrace 1,\cdots, h\rbrace$, $r_k=d_{k+1}$. In particular $r_h=1$, and $r_1$ divides $r_0=n$.
    \item For all $k\in\lbrace 1,\cdots, h\rbrace, C(g_k)=0$.
\end{enumerate}
    
\end{proposicion}

\begin{proof} 1. The hypothesis implies that 

$$
\sum_{k=1}^h(e_k-1)r_k=n-1=\sum_{k=1}^h(e_k-1)d_{k+1}.
$$

\noindent On the other hand, for all $k\in\lbrace 1,\cdots,h\rbrace$, $d_{k+1}={\rm GCD}(r_k,d_k)$, whence $r_k\geq d_{k+1}$. Consequently, the above equality is true if and only if $r_k=d_{k+1}$. The other assertions follow immediately. 
\medskip 

\noindent 2. The $\underline{r}$-sequence associated with $g_h$ is  given by $(r_0/d_h,\cdots,r_{h-1}/d_h)$, whence $C(g_h)=\sum_{i=1}^{h-1}(e_i-1)(r_i/d_h)-(n/d_h)+1$. In particular $0=C(f)=d_hC(g_h)+(d_h-1)(r_h-1)$. This proves that $C(g_h)=0$. We prove in a similar way that $C(g_k)=0$ for all $k\in\lbrace 1,\cdots,h\rbrace$.
\end{proof}

\begin{corolario}\label{mu=0} Let the notations be as in Proposition \ref{aut}. If $C(f)=0$ then $f$ is equivalent to a coordinate, i.e. there exists an automorphism $\sigma$ of ${\mathbb K}[x,y]$ such that $\sigma\circ f$ is a coordinate of ${\mathbb K}[x,y]$. 
\end{corolario}

\begin{proof} Let $g_h={\rm App}(f,d_h)$. We have $\Gamma(g_h)=\langle r_0/d_2,\cdots,r_{h-1}/d_h=1\rangle$. Let

$$
f=g_h^{d_h}+\alpha^h_2g_h^{d_h-2}+\cdots+\alpha^h_{d_h}
$$

\noindent be the expansion of $f$ with respect to $g_h$. We have $\alpha^h_{d_h}\not=0$ and int$(f,\alpha_i)<ir_h$ for all $i$ such that $\alpha_i\not=0$. Note that int$(f,\alpha^h_i)\in\langle r_0,\cdots,r_{h-1}\rangle$. As $r_h=1$ and $i<e_h=d_h={\rm GCD}(r_0,\cdots,r_{h-1})$, it follows that int$(f,\alpha^h_i)=0$, whence $\alpha^h_i\in{\mathbb K}$. Now int$(f,\alpha^h_{d_h})=r_he_h=r_hd_h=d_h=r_{h-1}$, whence $\alpha^h_{d_h}=a_{h-1}g_{h-1}+b_{h-1}$ with $a_{h-1}\not=0$, and consequently 

$$
{\mathbb K}[f,g_h]={\mathbb K}[g_h,g_{h-1}].
$$

\noindent An induction argument shows that ${\mathbb K}[g_h,g_{h-1}]={\mathbb K}[g_{h-1},g_{h-2}]=\cdots ={\mathbb K}[x,y]$.
\end{proof}

\section{ Geometric applications} 

\noindent Let $f=y^n+a_1(x)y^{n-1}+\cdots+a_n(x)$ be a polynomial with one place at infinity, and suppose that deg$_xa_i(x)<i$ for all $i\in\lbrace 1,\cdots,n\rbrace$. Let $F(X,y)=f(X^{-1},y)$, and let $F_X$ (resp. $F_y$) be the $X$-derivative (resp. the $y$-derivative) of $F$. Let Root$(F)=\lbrace y_1,\cdots,y_n\rbrace$. We have

$$
F(X,y)=\prod_{k=1}^n(y-y_k),
$$

\noindent hence $F_y(x,y)=\sum_{k=1}^n\prod_{i\not=k}(y-y_i)$. In particular

$$
{\rm Int}(F,F_y)=nO_XF_y(X,y_1)=nO_X(\prod_{i=2}^n(y_1-y_i))
$$

\noindent and it follows from Lemma \ref{contact} that ${\rm Int}(F,F_y)=\sum_{k=1}^h(d_k-d_{k+1})m_k$. On the other hand, easy calculations show that $\sum_{k=1}^h(d_k-d_{k+1})m_k=\sum_{k=1}^h(e_k-1)R_k$. If $f_x$ (resp. $f_y$) denotes the $x$-derivative (resp. the $y$-derivative of $f$), then it follows that

$$
{\rm int}(f,f_y)=\sum_{k=1}^h(e_k-1)r_k.
$$

\noindent Let $\gamma(x)$ be a root of $F_y(X,y)=0$. By the chain rule of derivatives, we have

$$
\dfrac{d}{dX}F(X,\gamma)=\dfrac{dF}{dX}(X,\gamma)+\dfrac{dF}{dy}(X,\gamma)\gamma'(X)=\dfrac{dF}{dX}(X,\gamma),
$$

\noindent which implies that $O_XF(X,\gamma)-1=O_XF_X(X,\gamma)-1$. Adding this equality among the set of $(n-1)$ roots of $F_y$ we get 

$$
{\rm Int}(F,F_y)={\rm Int}(F_X,F_y)+n-1.
$$

\noindent In particular int$(f,f_y)={\rm int}(f_x,f_y)+n-1$. We finally get 

$$
{\rm int}(f_x,f_y)=\sum_{k=1}^h(e_k-1)r_k-n+1
$$

\noindent which is nothing but the conductor $C(S)$ of $S=\Gamma(f)$. 

\noindent Set $\mu={\rm int}(f_x,f_y)$. We call $\mu$ the Milnor number of the family $(f_{\lambda})_{\lambda}$. Whence,  
$\mu=C(S)=F(S)+1$. With these notations, we get the following.



\begin{proposicion}\label{aut1} Let the notations be as above, and suppose that $\mu=0$. For all $k\in\lbrace 1,\cdots,h\rbrace$, let $g_k={\rm App}(f,d_k)$. The following conditions hold.

\begin{enumerate}
\item For all $k\in\lbrace 1,\cdots, h\rbrace$, $r_k=d_{k+1}$. In particular $r_h=1$, and $r_1$ divides $r_0=n$.
    \item For all $k\in\lbrace 1,\cdots, h\rbrace, \mu_k={\rm int}((g_k)_x,(g_k)_y)=0$.
    \item The polynomial $f$ is equivalent to a coordinate, i.e. there exists an automorphism $\sigma$ of ${\mathbb K}[x,y]$ such that $\sigma\circ f$ is a coordinate of ${\mathbb K}[x,y]$. 
\end{enumerate}
    
\end{proposicion}

\begin{proof} This is nothing but Proposition \ref{aut} and Corollary \ref{mu=0}
\end{proof}

\noindent Let $p\in C_f$, and suppose, after possibly a change of coordinates, that $p=(0,0)$. We define the local Milnor number of $f$ at $p$, denoted $\mu_p(f)$, to be the dimension of the ${\mathbb K}$-vector space ${\mathbb K}[[x,y]]/(f_x,f_y)$. In particular $C_f$ is singular (resp. smooth) at $p$ if $\mu_p(f)>0$ (resp. $\mu_p(f)=0$). We define the Milnor number of $f$, denoted $\mu(f)$, to be the sum $\sum_{p\in C_f}\mu_p(f)$. We say that $C_f$ is smooth if $\mu(f)=0$ (equivalently, $C_f$ is smooth at each of its points). We also have $\mu=\sum_{\lambda}\mu(f_{\lambda})$. In particular, if $\mu=0$, then $C_{f_{\lambda}}$ is smooth for all $\lambda\in{\mathbb K}$. In particular, we get a geometric version of Proposition \ref{aut1}.

\begin{proposicion} Let the notations be as above. If the family $(f_{\lambda})_{\lambda}$ is smooth, then $f$ is equivalent to a coordinate.
    
\end{proposicion}
\noindent Next, we will focus on rational curves with one place at infinity.
\subsection { Rational one place curves.}
\noindent Let $f=y^n+a_1(x)y^{n-1}+\cdots+a_n(x)$ be a polynomial with one place at infinity and suppose that deg$_xa_i(x)<k$ for all $i$ such that $a_i(x)\not=0$. Let the notations be as in the introduction. In particular, if $p\in C_f$, then $\xi_p(f)$ denotes the number of places of $f$ at $p$. Let $p\in C_f$. The order of the conductor of $C_f$ at $p$ is defined to be the codimension of the local ring of $C_f$ at $p$ in its integral closure. We denote it by $\delta_p$. We have $2\delta_p=\mu_p(f)+\xi_p(f)-1$ (see \cite{m}). Let $f_{\infty}=h_f(1,y,u)\in{\mathbb K}[[u,y]]$ be the local equation of $f$ at the point at infinity. We denote by $\delta_{\infty}$ the order of the conductor of $C_f$ at the point at infinity. As $\xi_{\infty}(f)=1$, we have $2\delta_{\infty}=\mu_{\infty}$, where $\mu_{\infty}$ denotes the local Milnor number at the point at infinity (which is the dimension of the ${\mathbb K}$-vector space ${\mathbb K}[[u,y]]/((f_{\infty})_u,(f_{\infty})_y)$. If we set $2\delta=\sum_{p\in C_f}2\delta_p$, then, by the genus formula, we get
$$
2g(f)+2\delta+2\delta_{\infty}=(n-1)(n-2)
$$

\noindent where $g(f)$ denotes the genus of $C_f$. On the other hand, by Bézout theorem, $n(n-1)={\rm int}(f,f_y)+{\rm int}(f_{\infty},(f_{\infty})_y)=\mu+\mu_{\infty}+2(n-1)$, whence $\mu+\mu_{\infty}=(n-1)(n-2)$. Combining these equalities, we get

$$
2g(f)+\sum_{p\in C_f}(\mu_p+\xi_p(f)-1)+\mu_{\infty}=\mu+\mu_{\infty}.
$$

\noindent This implies $2g(f)+\mu(f)+\sum_{p\in C_f}(\xi_p(f)-1)=\mu$. In particular, 

$$
\mu-\mu(f)=2g(f)+\sum_{p\in C_f}(\xi_p(f)-1),
$$
\noindent whence, $\mu=\mu(f)$ if and only if $g(f)=0$ and for all $p\in C_f, \xi_p(f)=1$. Roughly speaking, if $f$ is rational and has one place at each point of $C_f$, then for all $\lambda\not=0$, $C_{f_{\lambda}}$ is a smooth curve. Conversely, if $C_f$ is a smooth curve, then we get $\mu=2g(f)$. In particular, the geometric genus of a smooth curve of the family $(f_{\lambda})_{\lambda}$ coincides with the genus of the semigroup $\Gamma(f)$ of $f$. We also have the following.


\begin{proposicion}\label{rational} With the notations above, if $C_f$ is smooth and rational, then $C_f$ is equivalent to a coordinate, i.e. there exists an automorphism 
    $\sigma$ of ${\mathbb K}[x,y]$ such that $\sigma\circ f$ is a coordinate of ${\mathbb K}[x,y]$. In particular, for all $\lambda\in{\mathbb K}$, $V(f_{\lambda})$ is smooth and rational.
\end{proposicion}
\begin{proof} The hypotheses implies that $\mu=0$. Now we use Proposition \ref{aut1}.
\end{proof}

\noindent Next, we recall the result of S.S. Abhyankar and T.T. Moh.

\begin{proposicion} \label{A-M}(Abhyankar-Moh Theorem) Let $x(t),y(t)$ be two polynomials of ${\mathbb K}[t]$ and let $n={\rm deg}_tx(t)$, $m={\rm deg}_ty(t)$. Suppose that $m<n$. If ${\mathbb K}[x(t),y(t)]={\mathbb K}[t]$ then $m$ divides $n$, and $C_f$ is equivalent to a coordinate.
\end{proposicion}

\begin{proof} Let $f(x,y)$ be the generator, monic in $y$, of the kernel of the map $\phi: {\mathbb K}[x,y]\longmapsto {\mathbb K}[t], \phi(x)=x(t),\phi(y)=y(t)$. Then $f$ has one place at infinity. Moreover, the hypothesis shows that $C_f$ is smooth. The result is a consequence of Proposition \ref{aut}.
    \end{proof}

\noindent Abhyankar-Moh Theorem has been generalized by V. Lin and M. Zaidenberg as follows.

\begin{teorema}(see \cite{LZ}) Let $f$ be a polynomial with one place at infinity, and assume that $\xi_p(f)=1$ for all $p\in C_f$, then $f$ is equivalent to a quasi-homogeneous polynomial. In particular, if $C_f$ is smooth, then $f$ is equivalent to a coordinate. 
    
\end{teorema}

\noindent {\bf Problem.} The proof of Lin-Zaidenberg Theorem uses geometric and topological results, and one may ask if there exists an algebraic proof that uses the properties of semigroups associated with curves with one place at infinity (see \cite{a-s} for some partial results).

\medskip





\noindent Next, we show that a family $(f_{\lambda})_{\lambda}$ of curves with one place at infinity has at most two rational elements (see \cite{assi5}). Let $p$ be a point of $C_f$. If $\mu_p(f)=0$, then $\xi_p(f)=1$, whence $\mu_p(f)=\xi_p(f)-1$. Suppose that $p=(0,0)$, and also that $\mu_p>0$. Let $\tilde{f}=y^N+b_1(x)y^{N-1}+\cdots+b_N(x)\in{\mathbb K}[[x]][y]$ be the Weierstrass polynomial of $C_f$ at $p$, and let $\tilde{f}=\prod_{k=1}^{\xi_p(f)}\tilde{f}_k$ be the decomposition of $\tilde{f}$ into irreducible components in ${\mathbb K}[[x]][y]$ (we recall that the number of irreducible components of $\tilde{f}$ is nothing but $\xi_p(f)$. We have int$(\tilde{f},\tilde{f}_y)=\sum_{k=1}^{\xi_p(f)}{\rm int}(\tilde{f}_k,(\tilde{f}_k)_{y})+2\sum_{1\leq i<j\leq \xi_p(f)}{\rm int}(\tilde{f}_i,\tilde{f}_j)$. If $N_k={\rm deg}_y\tilde{f}_k$ for all $k\in\lbrace 1,\cdots,\xi_p{f}\rbrace$, then int$(\tilde{f}_k,(\tilde{f}_k)_{y})=\mu_p(f_k)+N_k-1$. As int$(\tilde{f},(\tilde{f})_y)=\mu_p(\tilde{f})+N-1=\mu_p(f)+N-1$, we have
$$
\mu_p(f)+N-1={\rm int}(\tilde{f},\tilde{f}_y)=\sum_{k=1}^{\xi_p(f)}(\mu_p(\tilde{f}_k)+N_k-1)+2\sum_{1\leq i<j\leq \xi_p(f)}{\rm int}(\tilde{f}_i,\tilde{f}_j).
$$

\noindent Consequently, as $N=\sum_{k=1}^{\xi_p(f)}N_k$, 

$$
\mu_p(f)=(\sum_{k=1}^{\xi_p(f)}\mu_p(\tilde{f}_k))-\xi_p(f)+1+2\sum_{1\leq i<j\leq \xi_p(f)}{\rm int}(\tilde{f}_i,\tilde{f}_j).
$$

\noindent Note that in the second summation, ${\rm int}(\tilde{f}_i,\tilde{f}_j)\geq 1$ for all $1\leq i<j\leq \xi_p(f)$, whence $\sum_{1\leq i<j\leq \xi_p(f)}{\rm int}(\tilde{f}_i,\tilde{f}_j)\geq \xi_p(f)(\xi_p(f)-1)/2$. In particular, 

$$
\mu_p(f)\geq (\sum_{k=1}^{\xi_p(f)}\mu_p(\tilde{f}_k))+(\xi_p(f)-1)^2.
$$



\noindent It follows from the equality above that $\mu_p(f)\geq (\xi_p(f)-1)^2\geq \xi_p(f)-1$, whence $\mu(f)=\sum_{p\in C_f}\mu_p(f)\geq \sum_{p\in C_f}(\xi_p(f)-1)^2$. In particular $\mu-\mu(f)\leq 2g(f)+\mu(f)$, and it follows that

$$
\mu\leq 2g(f)+2\mu(f).
$$

\noindent If furthermore $g(f)=0$, then $\mu(f)\geq \mu/2$. As a corollary we get the following.

\begin{corolario} With the notations above, if $g(f)=0$, then the family $(f_{\lambda})_{\lambda}$ has at most one other rational singular element.
    \end{corolario}
    
\noindent Remark that if $\xi_p(f)\geq 3$, then  $\mu_p(f)\geq (\xi_p(f)-1)^2> \xi_p(f)-1$, whence $\mu(f)>\mu/2$, in particular we have the following.

\begin{teorema} (see \cite{assi5}) With the notations above, suppose that $g(f)=0$, and also that $C_f$ is singular. if $\mu(f)=\mu/2$, then for all singular point $p$ of $C_f$, $\xi_p(f)=2$, and int$(f_1,f_2)=1$, whence $p$ is a double point of $C_f$. Moreover, $C_f$ has exactly $\mu/2$ double points.
    
\end{teorema}  

\begin{exemple} Let $n\geq 3$ and $x(t)=t^n-t, y(t)=t^{n-1}$. Let $f(x,y)$ be the kernel of the map $\phi:{\mathbb K}[x,y]\longmapsto {\mathbb K}[t], \phi(x)=x(t),\phi(y)=y(t)$. Then $f$ has one place at infinity, and $x'(t)=y'(t)=0$ does not have solutions in ${\mathbb K}$, whence all singular points of $C_f$ are double points. 
    
\end{exemple}
\begin{nota} Let $x(t)=t^3-3t$ and $y(t)=t^2-2$. The kernel of the map $\phi:{\mathbb K}[x,y]\longmapsto {\mathbb K}[t], \phi(x)=x(t),\phi(y)=y(t)$ is $f(x,y)=y^3-x^2-3y+2$, and $\nu=2$. Moreover, $\mu(f)=1=\mu/2$. If $g(x,y)=y^3-x^2-3y-2=f-4$, then $\mu(g)=1=\mu/2$, and $g$ is parametrized by $x=t^3+3t, y=t^2+2$. In this example, the family $(f_{\lambda})_{\lambda}$ has exactly two rational singular elements. We do not have similar examples when the degree is $\geq 4$, and we think that the family $(f_{\lambda})_{\lambda}$ has at most one rational singular element in this case.
    
\end{nota}

\noindent The above remark suggests the following question. Let $f$ be a nonzero polynomial of ${\mathbb K}[x,y]$ with one place at infinity. Does the family contain a rational element? The answer to this question is no. Let $f=(y^3-x^2)^2-4x$. Then the $\underline{r}$-sequence associated with $f$ is $(r_0=6,r_1=4,r_2=3)$, and the genus of $C_f$ is $\mu/2=3$. The pencil $(f_{\lambda})_{\lambda}$ has three singular elements, $f+3,f+3\alpha,f+3\alpha^2$, where $\alpha$ is a primitive cube root of unity in ${\mathbb K}$, and each of them had genus $1$. Thus the family does not contain rational elements.

\medskip

\begin{nota} Let $F,G\in {\mathbb K}[x,y]$. We say that $C_F$ and $C_G$ are isomorphic if their rings of coordinates are isomorphic. We say that $C_F$ and $C_G$ are equivalent if there exists an automorphism $\sigma:{\mathbb K}^2\longmapsto {\mathbb K}^2$ such that $\sigma(C_F)=C_G$. With these notations, Proposition \ref{A-M} says that if a rational smooth curve $C_f$ with one place at infinity is isomorphic to a coordinate, then $C_f$ is equivalent to a coordinate, i.e. the reduced $\underline{r}$ sequence of $f$ is $(r_0=1)$. This is not true for rational curves with one place at infinity. Let $x=t^4-t,z=t^3-1$, and $y=z^2+z/2=t^6-(3/2)t^3+1/2$. Clearly ${\mathbb K}[x,y]\subseteq {\mathbb K}[x,z]$. On the other hand, as $x=tz$, we have $x^3=t^3z^3=(t^3-1+1)z^3=z^4+z^3$, whence, as $y^2=z^4+z^3+(1/4)z^2$, $z^2=4(x^3-y^2)\in{\mathbb K}[x,y]$. In particular, $z=2(y-z^2)\in{\mathbb K}[x,y]$, and consequently ${\mathbb K}[x,y]={\mathbb K}[x,z]$. Now, if $f$ (resp. $g$) is the monic generator of the map $\phi:{\mathbb K}[x,y]\longmapsto {\mathbb K}[t]$ (resp. $\phi:{\mathbb K}[x,z]\longmapsto {\mathbb K}[t]$) such that $\phi(x)=x(t), \phi(y)=y(t)$ (resp. $\psi(x)=x(t), \psi(z)=z(t)$), then $f=z^4+z^3-x^3$, and $g=y^4-(1/2)y^3-2x^3y^2+(1/2)x^3y+x^6+x^3=(y^2-x^3)^2-(1/2)y(y^2-x^3)+x^3$, and the $\underline{r}$-sequence associated with $f$ (resp. $g$) is $(4,3)$ (resp. $(6,4,3)$), whence $C_f$ and $C_g$ are not equivalent. This example answers negatively a conjecture proposed by V. ShpilrainU and J.-T. Yu in \cite{V-Y}. It has been communicated to us by A. Sathaye.
\end{nota}

\noindent In connection with this example, one may ask the following.

\medskip 

\noindent {\bf Question.} Which $\underline{r}$-sequences are associated with rational curves with one place at infinity?







\end{document}